%% file: samplepaper.tex
\documentclass[runningheads]{llncs}
\usepackage[T1]{fontenc}
\usepackage{graphicx}
\usepackage[dvipsnames]{xcolor}
\usepackage{amsmath}
\usepackage{amssymb}

\usepackage{algorithm}
\usepackage[noend]{algpseudocode}
\usepackage{babel}
\usepackage{soul}
\usepackage{subcaption}
\usepackage{comment}
\usepackage{soul}

\usepackage[title]{appendix}
\usepackage{tikz}
\usepackage{listings}

\definecolor{codegreen}{rgb}{0,0.6,0}
\definecolor{codegray}{rgb}{0.5,0.5,0.5}
\definecolor{codepurple}{rgb}{0.58,0,0.82}
\definecolor{backcolour}{rgb}{0.95,0.95,0.92}

\lstdefinestyle{mystyle}{
  backgroundcolor=\color{backcolour},   commentstyle=\color{codegreen},
  keywordstyle=\color{magenta},
  numberstyle=\tiny\color{codegray},
  stringstyle=\color{codepurple},
  basicstyle=\ttfamily\footnotesize,
  breakatwhitespace=false,         
  breaklines=true,                 
  captionpos=b,                    
  keepspaces=true,                 
  numbers=left,                    
  numbersep=5pt,                  
  showspaces=false,                
  showstringspaces=false,
  showtabs=false,                  
  tabsize=2
}

\begin{document}
\title{Generating 2-Gray codes for grand Motzkin paths and grand Dyck paths with air pockets in constant amortized time}
\titlerunning{Generating 2-Gray codes for GMAP and GDAP in CAT}

\renewcommand{\footnoterule}{\vspace{2pt}\hrule width\textwidth height0.4pt\vspace{2pt}}

% \author{Lei Dong\inst{1}^{,}\inst{2} \and Bowie Liu \inst{2} \and Dennis Wong \inst{2} \and Lin Chen \inst{1}^{,}\inst{2} \and  \\ Chan-Tong Lam \inst{2} \and Sio-Kei Im \inst{2}}
\author{Lei Dong\inst{1,2} \and Bowie Liu\inst{2} \and Dennis Wong\inst{2} \and Lin Chen\inst{1,2} \and \\ Chan-Tong Lam\inst{2} \and Sio-Kei Im\inst{2}}
%\author{Lei Dong\inst{1}\orcidID{0000-1111-2222-3333} \and Bowie Liu \inst{1} \and Dennis Wong \inst{1} \and Lin Chen \inst{1} \and Chan-Tong Lam \inst{1} \and Sio-Kei Im \inst{2}}
%
\authorrunning{Lei Dong et al.}
\institute{School of Computer Science and Engineering,
Sun Yat-sen University, China
\email{chenlin69@mail.sysu.edu.cn}\\
%\url{http://www.springer.com/gp/computer-science/lncs}  
\and
Faculty of Applied Sciences, Macao Polytechnic University, Macao
\\ \email{\{lei.dong, bowen.liu, cwong, lchen, ctlam, marcusim\}@mpu.edu.mo}
}
%\institute{Faculty of Applied Science, Macao Polytechnic University, Macao
%\email{cwong@mpu.edu.mo}\\
%Macao Polytechnic University, Macao\\
\maketitle              
% typeset the header of the contribution

\vspace{-3em} 
\vspace{2em}   
\begin{abstract}
A {grand Motzkin path with air pockets} is a non-empty lattice path in the first and fourth quadrant of \(\mathbb{Z}^2\), starting at the origin \((0,0)\), ending on the \(x\)-axis, and consisting of up-steps \((1, 1)\), horizontal steps \((1, 0)\), down-steps \((1, -k)\) where \(k \geq 1\), and with no consecutive down-steps. A {grand Dyck path with air pockets} is a grand Motzkin path with air pockets that uses no horizontal steps. We present the first known 2-Gray codes for grand Motzkin paths with air pockets. Setting the number of horizontal steps to zero in our algorithm yields the first known 2-Gray codes for grand Dyck paths with air pockets.
Our three-stage algorithm generates each path in constant amortized time per string, using \(O(n^2)\) memory. We also provide enumeration formulae for grand Motzkin paths and grand Dyck paths with air pockets.

\end{abstract}

\section{Introduction}

A \emph{grand Motzkin path with air pockets} is a non-empty lattice path in the first and fourth quadrant of \(\mathbb{Z}^2\), starting at the origin \((0,0)\), ending on the \(x\)-axis, and consisting of up-steps \((1, 1)\), horizontal steps \((1, 0)\), and down-steps \((1, -k)\) for \(k \geq 1\), with the condition that no two down-steps are consecutive.
As an example, Figure~\ref{fig:M&D} illustrates the ten grand Motzkin paths with air pockets for \(n = 3\). 
Each grand Motzkin path with air pockets can be represented by a unique tuple $(a_1, a_2, \cdots, a_n)$, where \(a_i\) denotes the step at position \(i\) in the path, with \(a_i = 1\) representing an up-step \((1, 1)\), \(a_i = 0\) representing a horizontal step \((1, 0)\), and \(a_i = -k\) representing a down-step \((1, -k)\) for \(k \geq 1\). 
The tuples corresponding to the ten grand Motzkin paths with air pockets in Figure~\ref{fig:M&D} for \(n = 3\) are as follows:
\begin{center}
\begin{tabular}{l}
$(-2, 1, 1), (-1, 0, 1), (-1, 1, 0),
(0, -1, 1), (0, 0, 0),$ \\
$(0, 1, -1), (1, -2, 1), (1, -1, 0),
(1, 0, -1), (1, 1, -2).$
\end{tabular}
\end{center}
A \emph{grand Dyck path with air pockets} is a grand Motzkin path with air pockets that uses no horizontal steps. 
As an example, the tuples corresponding to the three grand Dyck paths with air pockets (underlined in Figure~\ref{fig:M&D}) for $n = 3$ are $(-2, 1, 1)$,  $(1, -2, 1)$,  and $(1, 1, -2)$.

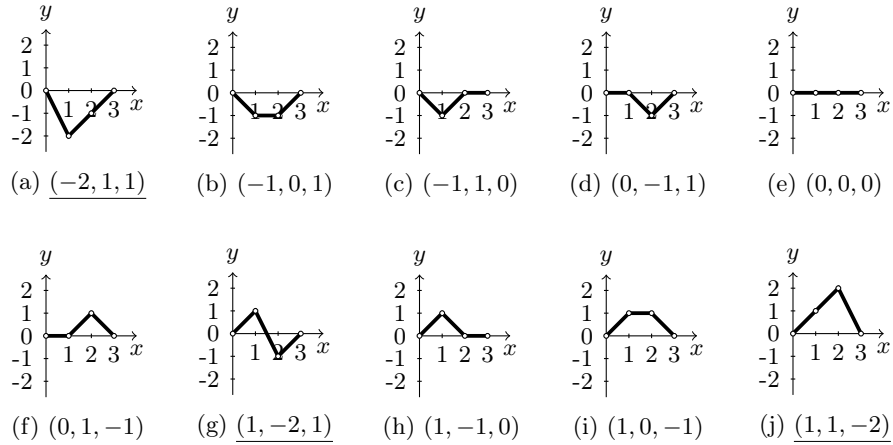
\begin{figure}[t]
% First row with 5 subfigures
\centering
\begin{subfigure}{0.19\textwidth}
    \centering
    \begin{tikzpicture}[scale=0.3]
        \draw[->] (0,0) -- (4,0) node[below] {$x$};
        \draw[->] (0,-2.7) -- (0,2.7) node[above] {$y$};
        \foreach \x in {1,2,3} \draw (\x,0.1) -- (\x,-0.1) node[below] {\x};
        \foreach \y in {-2,-1,0,1,2} \draw (0.1,\y) -- (-0.1,\y) node[left=0.03cm] {\y};
        \draw[black, line width=0.5mm] (0,0) -- (1,-2) -- (2,-1) -- (3,0);
        \foreach \x/\y in {0/0, 1/-2, 2/-1, 3/0} \draw[draw=black, fill=white] (\x,\y) circle (3pt);
    \end{tikzpicture}
    \caption{\underline{\color{black}$(-2, 1, 1)$}}
\end{subfigure}\hfill
\begin{subfigure}{0.19\textwidth}
    \centering
    \begin{tikzpicture}[scale=0.3]
        \draw[->] (0,0) -- (4,0) node[below] {$x$};
        \draw[->] (0,-2.7) -- (0,2.7) node[above] {$y$};
        \foreach \x in {1,2,3} \draw (\x,0.1) -- (\x,-0.1) node[below] {\x};
        \foreach \y in {-2,-1,0,1,2} \draw (0.1,\y) -- (-0.1,\y) node[left=0.03cm] {\y};
        \draw[black, line width=0.5mm] (0,0) -- (1,-1) -- (2,-1) -- (3,0);
        \foreach \x/\y in {0/0, 1/-1, 2/-1, 3/0} \draw[draw=black, fill=white] (\x,\y) circle (3pt);
    \end{tikzpicture}
    \caption{$(-1, 0, 1)$}
\end{subfigure}\hfill
\begin{subfigure}{0.19\textwidth}
    \centering
    \begin{tikzpicture}[scale=0.3]
        \draw[->] (0,0) -- (4,0) node[below] {$x$};
        \draw[->] (0,-2.7) -- (0,2.7) node[above] {$y$};
        \foreach \x in {1,2,3} \draw (\x,0.1) -- (\x,-0.1) node[below] {\x};
        \foreach \y in {-2,-1,0,1,2} \draw (0.1,\y) -- (-0.1,\y) node[left=0.03cm] {\y};
        \draw[black, line width=0.5mm] (0,0) -- (1,-1) -- (2,0) -- (3,0);
        \foreach \x/\y in {0/0, 1/-1, 2/0, 3/0} \draw[draw=black, fill=white] (\x,\y) circle (3pt);
    \end{tikzpicture}
    \caption{$(-1, 1, 0)$}
\end{subfigure}\hfill
\begin{subfigure}{0.19\textwidth}
    \centering
    \begin{tikzpicture}[scale=0.3]
        \draw[->] (0,0) -- (4,0) node[below] {$x$};
        \draw[->] (0,-2.7) -- (0,2.7) node[above] {$y$};
        \foreach \x in {1,2,3} \draw (\x,0.1) -- (\x,-0.1) node[below] {\x};
        \foreach \y in {-2,-1,0,1,2} \draw (0.1,\y) -- (-0.1,\y) node[left=0.03cm] {\y};
        \draw[black, line width=0.5mm] (0,0) -- (1,0) -- (2,-1) -- (3,0);
        \foreach \x/\y in {0/0, 1/0, 2/-1, 3/0} \draw[draw=black, fill=white] (\x,\y) circle (3pt);
    \end{tikzpicture}
    \caption{$(0, -1, 1)$}
\end{subfigure}\hfill
\begin{subfigure}{0.19\textwidth}
    \centering
    \begin{tikzpicture}[scale=0.3]
        \draw[->] (0,0) -- (4,0) node[below] {$x$};
        \draw[->] (0,-2.7) -- (0,2.7) node[above] {$y$};
        \foreach \x in {1,2,3} \draw (\x,0.1) -- (\x,-0.1) node[below] {\x};
        \foreach \y in {-2,-1,0,1,2} \draw (0.1,\y) -- (-0.1,\y) node[left=0.03cm] {\y};
        \draw[black, line width=0.5mm] (0,0) -- (1,0) -- (2,0) -- (3,0);
        \foreach \x/\y in {0/0, 1/0, 2/0, 3/0} \draw[draw=black, fill=white] (\x,\y) circle (3pt);
    \end{tikzpicture}
    \caption{$(0, 0, 0)$}
\end{subfigure}

\vspace{0.5cm}

% Second row with 5 subfigures
\begin{subfigure}{0.19\textwidth}
    \centering
    \begin{tikzpicture}[scale=0.3]
        \draw[->] (0,0) -- (4,0) node[below] {$x$};
        \draw[->] (0,-2.7) -- (0,2.7) node[above] {$y$};
        \foreach \x in {1,2,3} \draw (\x,0.1) -- (\x,-0.1) node[below] {\x};
        \foreach \y in {-2,-1,0,1,2} \draw (0.1,\y) -- (-0.1,\y) node[left=0.03cm] {\y};
        \draw[black, line width=0.5mm] (0,0) -- (1,0) -- (2,1) -- (3,0);
        \foreach \x/\y in {0/0, 1/0, 2/1, 3/0} \draw[draw=black, fill=white] (\x,\y) circle (3pt);
    \end{tikzpicture}
    \caption{$(0, 1, -1)$}
\end{subfigure}\hfill
\begin{subfigure}{0.19\textwidth}
    \centering
    \begin{tikzpicture}[scale=0.3]
        \draw[->] (0,0) -- (4,0) node[below] {$x$};
        \draw[->] (0,-2.7) -- (0,2.7) node[above] {$y$};
        \foreach \x in {1,2,3} \draw (\x,0.1) -- (\x,-0.1) node[below] {\x};
        \foreach \y in {-2,-1,0,1,2} \draw (0.1,\y) -- (-0.1,\y) node[left=0.03cm] {\y};
        \draw[black, line width=0.5mm] (0,0) -- (1,1) -- (2,-1) -- (3,0);
        \foreach \x/\y in {0/0, 1/1, 2/-1, 3/0} \draw[draw=black, fill=white] (\x,\y) circle (3pt);
    \end{tikzpicture}
    \caption{\underline{\color{black}$(1, -2, 1)$}}
\end{subfigure}\hfill
\begin{subfigure}{0.19\textwidth}
    \centering
    \begin{tikzpicture}[scale=0.3]
        \draw[->] (0,0) -- (4,0) node[below] {$x$};
        \draw[->] (0,-2.7) -- (0,2.7) node[above] {$y$};
        \foreach \x in {1,2,3} \draw (\x,0.1) -- (\x,-0.1) node[below] {\x};
        \foreach \y in {-2,-1,0,1,2} \draw (0.1,\y) -- (-0.1,\y) node[left=0.03cm] {\y};
        \draw[black, line width=0.5mm] (0,0) -- (1,1) -- (2,0) -- (3,0);
        \foreach \x/\y in {0/0, 1/1, 2/0, 3/0} \draw[draw=black, fill=white] (\x,\y) circle (3pt);
    \end{tikzpicture}
    \caption{$(1, -1, 0)$}
\end{subfigure}\hfill
\begin{subfigure}{0.19\textwidth}
    \centering
    \begin{tikzpicture}[scale=0.3]
        \draw[->] (0,0) -- (4,0) node[below] {$x$};
        \draw[->] (0,-2.7) -- (0,2.7) node[above] {$y$};
        \foreach \x in {1,2,3} \draw (\x,0.1) -- (\x,-0.1) node[below] {\x};
        \foreach \y in {-2,-1,0,1,2} \draw (0.1,\y) -- (-0.1,\y) node[left=0.03cm] {\y};
        \draw[black, line width=0.5mm] (0,0) -- (1,1) -- (2,1) -- (3,0);
        \foreach \x/\y in {0/0, 1/1, 2/1, 3/0} \draw[draw=black, fill=white] (\x,\y) circle (3pt);
    \end{tikzpicture}
    \caption{$(1, 0, -1)$}
\end{subfigure}\hfill
\begin{subfigure}{0.19\textwidth}
    \centering
    \begin{tikzpicture}[scale=0.3]
        \draw[->] (0,0) -- (4,0) node[below] {$x$};
        \draw[->] (0,-2.7) -- (0,2.7) node[above] {$y$};
        \foreach \x in {1,2,3} \draw (\x,0.1) -- (\x,-0.1) node[below] {\x};
        \foreach \y in {-2,-1,0,1,2} \draw (0.1,\y) -- (-0.1,\y) node[left=0.03cm] {\y};
        \draw[black, line width=0.5mm] (0,0) -- (1,1) -- (2,2) -- (3,0);
        \foreach \x/\y in {0/0, 1/1, 2/2, 3/0} \draw[draw=black, fill=white] (\x,\y) circle (3pt);
    \end{tikzpicture}
    \caption{\underline{\color{black}$(1, 1, -2)$}}
\end{subfigure}

% Figure caption
\caption{Grand Motzkin paths with air pockets for \( n = 3 \). Grand Motzkin paths with air pockets that are underlined are grand Dyck paths with air pockets.}
\label{fig:M&D}
\end{figure}

The number of grand Dyck paths with air pockets, also known as \emph{Whitney numbers of the second kind}, corresponds to sequence A051291 in the Online Encyclopedia of Integer Sequences~\cite{Sloane2010}. 
Baril et al.~\cite{Baril23b,Baril23a,baril2024grand} studied enumerative results and derived generating functions for grand Dyck paths with air pockets under various restrictions. 
For example, the number of grand Dyck paths with air pockets of length $n$ is given by the following generating function:
\begin{equation}\nonumber
G(t) = \frac{1 + 2t^2 - t^3 - (1 - t) \sqrt{1 - 2t - t^2 - 2t^3 + t^4}}{2t \sqrt{1 - 2t - t^2 - 2t^3 + t^4}}.
\end{equation}
The first ten terms of this generating function are 1, 2, 3, 7, 17, 40, 97, 238, 587, and 1458. 
However, no closed-form formula is known for this sequence, and evaluating the generating function for large $n$ is computationally expensive due to the complex square root term. 
The first ten terms for the enumeration sequence of grand Motzkin paths with air pockets are 1, 3, 10, 32, 106, 361, 1247, 4351, 15304, and 54172. 
More interestingly, the number of grand Motzkin paths with air pockets is not even listed in the Online Encyclopedia of Integer Sequences as of this writing, and no closed-form formula has been established. 
This paper presents the first known closed-form formulae for the enumeration of both grand Dyck paths and grand Motzkin paths with air pockets.

Dyck paths and Motzkin paths are well-studied combinatorial objects with a wide variety of applications in areas such as geometry, algebra, probability, and computer science. 
%Knuth dedicates fascicle 3A of \emph{The Art of Computer Programming, Volume 4A}~\cite{Art4} to combinations, which are in bijection with simple Dyck paths via the ballot theorem. 
For example, Dyck paths have been extensively studied in Knuth's \emph{The Art of Computer Programming, Volume 4, Fascicle 4}~\cite{Knuth06Fasc4}, particularly in the contexts of binary trees, parenthesis matchings, and lattice paths.
Motzkin paths, which generalize Dyck paths by allowing horizontal steps, appear in diverse applications including RNA secondary structure modeling, non-crossing partitions, and bounded lattice paths~\cite{goulden1983combinatorial,stanton1986constructive}. 
Both Dyck paths and Motzkin paths encode a wide variety of combinatorial objects, including binary trees, balanced parentheses, ballot sequences, and stack-sortable permutations \cite{BENCHEKROUN1996605,DEUTSCH1998253,DEUTSCH2002655,LABELLE19901,Panayotopoulos1995,SabriVajnovszki+2019+109+119,viennot1980theorie,ballotsequence}.
%Dyck paths and Motzkin paths are well-studied combinatorial objects with a wide variety of applications in areas such as geometry, algebra, and probability. 
%Motzkin paths, in particular, appear in geometry, number theory, and combinatorics. 
%Knuth dedicated a whole chapter in The Art of Computer Programming~\cite{Art4} to combinations and their applications, which relate closely to Motzkin paths and Dyck paths. 
%Dyck paths and Motzkin paths encode a wide variety of combinatorial objects, including binary trees, balanced parentheses, ballot sequences, and stack-sortable permutations~\cite{BENCHEKROUN1996605,DEUTSCH1998253,DEUTSCH2002655,LABELLE19901,Panayotopoulos1995,SabriVajnovszki+2019+109+119,viennot1980theorie,ballotsequence}. 
An interesting application of grand Motzkin paths or grand Dyck paths with air pockets, as discussed in~\cite{baril2024grand}, is to model aviation paths where the ``air pockets'' represent turbulences with consecutive turbulences are merged into one. %models aviation paths where the "air pockets" represent turbulences and consecutive turbulences are merged into one. 
Grand Motzkin paths or grand Dyck paths with air pockets also have applications in queue or stack evolution, interpreting paths with resets or ``catastrophes''~\cite{banderier2017lattice}, and modeling stock market crashes.  
For more applications of Motzkin paths, Dyck paths, and related structures, see~\cite{goulden1983combinatorial,mutze2023combinatorial,Sav97,stanton1986constructive}.

%Dyck paths and Motzkin paths are well studied combinatorial objects that have a wide variety of applications.
%For example, Knuth dedicated a whole chapter in The Art of Computer Programming~\cite{Art4} to discuss combinations and their applications. 
%Dyck paths, on the other hand, have been used to encode a wide variety of combinatorial objects, including binary trees, balanced parentheses, ballot sequences, and stack-sortable permutations~\cite{BENCHEKROUN1996605, DEUTSCH1998253, DEUTSCH2002655, LABELLE19901, Panayotopoulos1995, SabriVajnovszki+2019+109+119, viennot1980theorie, ballotsequence}. 
%\hl{An interesting application} discussed in~\cite{} on grand Motzkin path or grand Dyck paths with air pockets is about aviation paths, where "air pockets" represent turbulences, and consecutive turbulences are merged into one.
%Grand Motzkin path or grand Dyck paths with air pockets also has applications in queue or stack evolution interpretation of paths with "catastrophes" or resets~\cite{}. 
%For more applications of combinations and Dyck paths, see~\cite{MR702512, mutze2022combinatorial, Sav97, stanton2012constructive}. 

One of the most important aspects of combinatorial generation is to list the instances of a combinatorial object so that consecutive instances differ by a specified \emph{closeness condition} involving a constant amount of change. 
Lists of this type are called \emph{Gray codes}. 
This terminology is due to the eponymous \emph{binary reflected Gray code} (BRGC) by Frank Gray, which orders the $2^n$ binary strings of length $n$ so that consecutive strings differ in one bit. 
For example, when $n = 5$ the order is
\begin{equation}
\begin{tabular}{l}
\label{comb_order_n5}00000, 00001, 00011, 00010, 00110, 00111, 00101, 00100,\\
01100, 01101, 01111, 01110, 01010, 01011, 01001, 01000,\\
11000, 11001, 11011, 11010, 11110, 11111, 11101, 11100,\\
10100, 10101, 10111, 10110, 10010, 10011, 10001, 10000.
\end{tabular}
\end{equation}
The BRGC listing is a \emph{1-Gray code} since consecutive strings differ by one bit change. 
%We note that the order is also \emph{cyclic} because the last and first
%strings also differ by the closeness condition, and this property holds for all $n$.
In this paper, we are focusing on \emph{2-Gray code}, where consecutive strings differ by at most two symbol changes.

An interesting related problem is thus to discover Gray codes for grand Motzkin paths and grand Dyck paths with air pockets and their relatives. 
%It is worth noting that there is no 1-Gray code for grand Dyck paths and grand Motzkin paths with air pockets. 
%For example, given a grand Dyck path or grand Motzkin path with air pockets, updating a up-step or a horizontal-step would require an update on a down-step since the path has to end on the x-axis, and thus requiring at least two symbols change. 
It is worth noting that there is no 1-Gray code for grand Dyck paths or grand Motzkin paths with air pockets. 
For example, in such a path, changing an up-step would require at least a compensating change to a down-step or a horizontal step to preserve the ending point on the $x$-axis. 
The same holds when updating a down-step or a horizontal step. Thus, at least two symbol changes are needed to reach another valid path.
For simple grand Dyck paths with downstep restricted to $k=1$, each such path can be represented by a combination with $n$ 1s (representing up-steps) and $n$ 0s (representing down-steps) among $2n$ total steps.
%n positions (for the 1s) among the 
%with 1 representing up-steps and 0 representing down-steps. 
Several algorithms have been proposed to generate combinations \cite{360343,knuth1970combinations,Art4,merino2022combinatorial,mutze2023combinatorial,ruskey1988generating,ruskey1996combinatorial}. 
The \emph{weight} of a binary string $\alpha$, denoted by $|\alpha|$, is defined as the number of 1s it contains.
The BRGC listing can be used to create a 2-Gray code for combinations. 
The Gray code is created by \emph{filtering} the BRGC,  which retains only strings in the listing with weight $ k $, and is referred to as the \emph{revolving door} Gray code for combinations. %~\cite{Wilf}. 
For example, filtering the BRGC listing in~(\ref{comb_order_n5}) gives the following 2-Gray code for combinations with $n=5$ and $k=2$:  
\begin{equation} \label{eq:RevolvingDoor42}\nonumber
\begin{aligned}
00011, 00110, 00101, 01100, 01010, 01001, 11000, 10100, 10010, 10001.
\end{aligned}
\end{equation}
For more information about Gray codes induced by the BRGC, see~\cite{10.1007/978-3-030-85088-3_15} and~\cite{flipswap}. 
%Grand Motzkin paths, on the other hand, are an extension of grand Dyck paths which allow horizontal steps. 
Several algorithms have also been proposed to generate grand Motzkin paths with downstep restricted to $k=1$, with constructions based on ECO methods or recursive decompositions~\cite{barcucci2000eco,Prodinger2021}.
Baril et al.~\cite{Baril23b,Baril23a,baril2024grand} first introduced grand Dyck paths with air pockets and provided generating functions of grand Dyck paths under various constraints~\cite{BarilBK20,BarilBousquetMelouKirgizovNaima2025,BarilKirgizov2021,BarilKirgizov2022,baril2024grand}. 
Subsequent work by Baril and Kirgizov~\cite{BarilKirgizov2021} established bijections to pattern avoiding Dyck paths, and 
Baril et al. explored several variants of grand Motzkin paths~\cite{BarilBK20,BarilFlorezRamirez2024,BarilKirgizov2021,BarilKirgizovNaima2025,BarilKirgizovPetrossian2019,BarilKirgizovRamirezVillamizar2025,BarilPallo2014,BarilPallo2015,BarilPetrossian2015,BarilRamirez2025}, but no Gray codes have been discovered for them.
%but {no Gray codes were discovered}.
This paper thus {provides} the first known 2-Gray code for both grand Motzkin paths and grand Dyck paths with air pockets.
Since there is no 1-Gray code for grand Motzkin paths or grand Dyck paths with air pockets, our Gray codes are optimal.

%Baril et al.~\cite{Baril23b,Baril23a,baril2024grand} first introduced the idea of grand Dyck paths with air pockets, and have 
%provided generating functions for the number of grand Dyck paths with air pockets under various constraints, but no Gray codes were discovered.
%This paper thus provide the first known cyclic 2-Gray code for grand Motzkin paths and grand Dyck paths with air pockets.

The rest of the paper is outlined as follows. 
In Section~\ref{sec:motzkin}, we describe a three-stage  algorithm for generating a 2-Gray code for grand Motzkin paths with air pockets.
% , and prove their Gray code property. 
Then in Section~\ref{sec:dyck}, we extend our algorithm to generate a 2-Gray code for grand Dyck paths with air pockets. 
We present enumeration formulae for grand Motzkin paths and grand Dyck paths with air pockets and prove their correctness in Section~\ref{sec:enumeration}.

\section{Generating grand Motzkin paths with air pockets}\label{sec:motzkin}

In this section, we describe a three-stage algorithm for generating a 2-Gray code for grand Motzkin paths with air pockets. 
We will discuss extending this algorithm to generate a 2-Gray code for grand Dyck paths with air pockets in the next section. 
The construction of the Gray code for grand Motzkin paths with air pockets involves the following three stages:
\begin{enumerate}
\item For a grand Motzkin path with $d$ down-steps, first generate a Gray code for binary strings $q = q_1 q_2 \cdots q_n$ of length $n$ and weight $d$. Here, $q_i = 1$ indicates that position $i$ is a down-step, and $q_i = 0$ indicates that position $i$ is an up-step or a horizontal step;
\item Given a string $q = q_1 q_2 \cdots q_n$ with weight $d$ from the first stage, generate a Gray code for binary strings $r = r_1 r_2 \cdots r_{n-d}$, where $r_i = 1$ indicates that the $i$-th position in $q$ with $q_j = 0$ corresponds to an up-step, and $r_i = 0$ indicates a horizontal step;
\item Given the strings $q$ and $r$ from the first and second stages, compute the total number of up-steps, denoted by $s$, {which is equal to the weight of $r$}. 
Then, generate a Gray code to distribute the integer $s$ across the positions where $q_i = 1$ (i.e., the down-steps). 
Each down-step receives a positive integer $t \geq 1$, and the sum of these $d$ values must equal $s$.
\end{enumerate}
Finally, we combine these listings to construct a 2-Gray code for grand Motzkin paths with air pockets. 
For example, consider the grand Motzkin path with air pockets as shown in Figure~\ref{fig:motzkin_example}, which can be represented by the tuple $m = (-2, 1, 1, 1, 0, -2, 1, 1, 1, -1, 0, -1)$. 
Then, $q = 100001000101$ and $r = 11101110$. 
The total number of up-steps is six, and thus $s = 6$ which equals the weight of $r$. 
We distribute the integer $s=6$ across the four down-step positions: two to the first and sixth positions of $m$ (resulting in $-2$), and one to the tenth and twelfth positions of $m$ (resulting in $-1$).
In the following subsections, we detail the procedure for generating Gray codes for the combinatorial objects in each stage of constructing grand Motzkin paths with air pockets. 

% Creating the figure for the grand Motzkin path
\begin{figure}[t]
\centering
\begin{tikzpicture}[scale=0.6]
    % Grand Motzkin Path: [-2, 1, 1, 1, 0, -2, 1, 1, 1, -1, 0, -1]
    % Path: (0,0) -> (1,-2) -> (2,-1) -> (3,0) -> (4,1) -> (5,1) -> (6,-1) -> (7,0) -> (8,1) -> (9,2) -> (10,1) -> (11,1) -> (12,0)
    % y-values: 0, -2, -1, 0, 1, 1, -1, 0, 1, 2, 1, 1, 0
    % Min y = -2, Max y = 2
    \draw[->] (0,0) -- (13,0) node[right] {$x$}; % x-axis, length 13 for 12 steps
    \draw[->] (0,-2.5) -- (0,2.5) node[above] {$y$}; % y-axis, from -2.5 to 2.5
    % x-axis ticks and labels
    \foreach \x in {1,2,...,12} {
        \draw (\x,0.1) -- (\x,-0.1) node[below] {\x};
    }
    \draw (0,0.1) -- (0,-0.1) node[below left] {$O$}; % Origin labeled as O
    % y-axis ticks and labels
    \foreach \y in {-2,-1,1,2} {
        \draw (0.1,\y) -- (-0.1,\y) node[left] {\y};
    }
    % Path
    \draw[black, line width=0.5mm] (0,0) -- (1,-2) -- (2,-1) -- (3,0) -- (4,1) -- (5,1) -- (6,-1) -- (7,0) -- (8,1) -- (9,2) -- (10,1) -- (11,1) -- (12,0);
    \foreach \x/\y in {0/0, 1/-2, 2/-1, 3/0, 4/1, 5/1, 6/-1, 7/0, 8/1, 9/2, 10/1, 11/1, 12/0} \draw[draw=black, fill=white] (\x,\y) circle (3pt); % Hollow circles
\end{tikzpicture}
\caption{The grand Motzkin path with air pockets which corresponds to $m = (-2, 1, 1, 1, 0, -2, 1, 1, 1, -1, 0, -1)$.}
\label{fig:motzkin_example}
\end{figure}
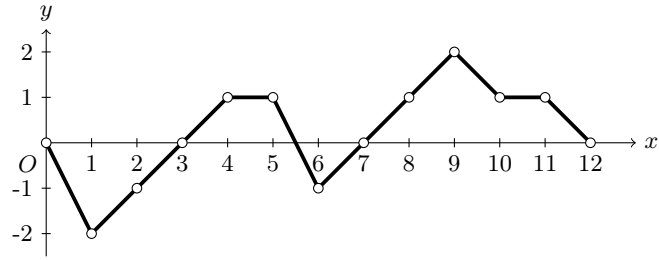

%In the following subsections, we detail the procedure for generating Gray codes for the combinatorial objects in each stage of constructing grand Motzkin paths with air pockets. 
%We then describe how to concatenate these listings to create a 2-Gray code for grand Motzkin paths with air pockets.

\subsection{Stage 1: Generating fixed weight Fibonacci words}

Given $d$ down-steps in a grand Motzkin path with air pockets, the first stage of our algorithm generates a 2-Gray code for binary strings $q$, where $q_i = 1$ indicates that position $i$ is a down-step, and $q_i = 0$ indicates an up-step or a horizontal step. 
Since grand Motzkin paths with air pockets have no consecutive down-steps, the string $q$ contains no consecutive 1s. 
A \emph{Fibonacci word} of order $p$, weight $k$, and length $n$ is a binary string of length $n$ with weight $k$ that avoids $p$ consecutive 1s. 
When $p = 2$, the set of Fibonacci words of weight $k$ and length $n$ corresponds to the set of binary strings of weight $k$ with no consecutive 1s, {which matches the set of strings $q$ in stage 1}  when we set $k = d$. 
In~\cite{Hassler24KMS,Hassler24GASCom,Hassler24RAIRO}, Hassler et al. developed efficient algorithms to generate greedy Gray codes for Fibonacci words of length $n$, order $p$, and weight $k$. 
Their recursive and greedy implementations generate each Fibonacci word in constant amortized time using $O(n^2)$ memory. 
The algorithm can be described recursively as follows when $p=2$ and $k=d$:
\begin{quote}
{Start} with $(10)^d 0^{n-2d}$. Greedily swap the leftmost possible $1$ with the leftmost possible $0$ before the next 1 and after the previous 1 (if there are any) such that the resulting string has not appeared before.
\end{quote}
Let \(\mathcal{F}_d(n)\) denote the Gray code listing of Fibonacci words of length $n$, order $p=2$, and weight $k=d$ generated by the algorithm in~\cite{Hassler24KMS,Hassler24GASCom,Hassler24RAIRO}.  %Let \(\mathcal{F}_d(n)\) denote the listing of strings of $q$ for grand Motzkin path with air pockets of length $n$ with $d$ down-steps generated by the algorithm in~\cite{Hassler24KMS,Hassler24GASCom,Hassler24RAIRO}.
For example, the algorithm generates the following listing \(\mathcal{F}_2(7)\) for Fibonacci words of weight $k=d=2$, length $n=7$, and order $p=2$:
\begin{center}
1010000, 1001000, 0101000, 0100100, 1000100, 
\\0010100, 0010010, 1000010, 0100010, 0001010, 
\\0001001, 1000001, 0100001, 0010001, 0000101.
\end{center}
This listing \(\mathcal{F}_2(7)\) corresponds to all possible arrangements of down-steps in a grand Motzkin path with air pockets for $n = 7$ and $d = 2$.

\begin{lemma}\label{lemma-start}~\cite{Hassler24KMS,Hassler24GASCom,Hassler24RAIRO}
The listing \(\mathcal{F}_d(n)\) always starts with the string $(10)^d 0^{n-2d}$ and ends with the string $ 0^{n-2d}(01)^d$. 
\end{lemma}

\begin{lemma}\label{lem:gray-Fd}~\cite{Hassler24KMS,Hassler24GASCom,Hassler24RAIRO}
Consecutive strings in \(\mathcal{F}_d(n)\) differ by at most two bit changes. 
\end{lemma}

\begin{lemma}\label{lem:CATFd}~\cite{Hassler24KMS,Hassler24GASCom,Hassler24RAIRO}
The listing  \(\mathcal{F}_d(n)\) can be generated in constant amortized time per string, using $O(n^2)$ memory. 
\end{lemma}

We then define $\mathcal{F}(n)$ as the listing produced by concatenating the listings $\mathcal{F}_d(n)$ for $d$ ranging from $\lfloor \frac{n}{2} \rfloor$ down to 1. However, we reverse each string in $\mathcal{F}_d(n)$ individually (i.e., take its mirror image) when $d$ is even.
We also use the notation $\mathcal{S}^{R}$ to refer to the listing obtained by reversing each string in a listing $\mathcal{S}$. For example, $(001, 200, 010)^R = (100, 002, 010)$.
When $n = 7$, we have
%By reversing the order of $\mathcal{F}_2(7)$ (since $\lfloor \frac{7}{2} \rfloor - 2 = 1$ is odd), we obtain:
%\begin{itemize}
%\end{itemize}
\begin{align*}
\mathcal{F}_3(7) &= 1010100, 1010010, 1001010, 0101010, 0101001, 1001001, 1010001, \\
& \ \ \ \  1000101, 0100101, 0010101;\\
\mathcal{F}_2(7)^{R} &= 0000101, 0001001, 0001010, 0010010, 0010001, 0010100, 0100100, 0100001, \\
&= 0100010, 0101000, 1001000, 1000001, 1000010, 1000100, 1010000;\\
%\mathcal{F}_2(7)^{-1} &= 0000101, 0010001, 0100001, 1000001, 0001001, 0001010, 0100010, 1000010, \\ 
%& \ \ \ \  0010010, 0010100, 1000100, 0100100, 0101000, 1001000, 1010000;\\
\mathcal{F}_1(7) &= 1000000, 0100000, 0010000, 0001000, 0000100, 0000010, 0000001.
\end{align*}
%\begin{itemize}\raggedright
%  \setlength{\itemsep}{2pt} 
%  \setlength{\parskip}{0pt} 
%  \item \(\mathcal{F}_3(7)\): {1010100, 1010010, 1001010, 0101010, 0101001, 1001001, 1010001, 1000101, 0100101, 0010101;}
%  \item \(\mathcal{F}_2(7)^{-1}\): {0000101, 0010001, 0100001, 1000001, 0001001, 0001010, 0100010, 1000010, 0010010, 0010100, 1000100, 0100100, 0101000, 1001000, 1010000;}
%  \item \(\mathcal{F}_1(7)\): {1000000, 0100000, 0010000, 0001000, 0000100, 0000010, 0000001.}
%\end{itemize}
Thus, $\mathcal{F}(7)$ is as follows:
\begin{center}
1010100, 1010010, 1001010, 0101010, 0101001, 1001001, 1010001, 1000101, 0100101, 0010101,
0000101, 0001001, 0001010, 0010010, 0010001, 0010100, 0100100, 0100001, 
0100010, 0101000, 1001000, 1000001, 1000010, 1000100, 1010000, 
1000000, 0100000, 0010000, 0001000, 0000100,
0000010, 0000001.
\end{center}

\begin{lemma}\label{Fn-Gray}
Consecutive strings in $\mathcal{F}(n)$ differ by at most two bit changes. 
\end{lemma}

\begin{proof}
By Lemma~\ref{lem:gray-Fd}, consecutive strings within each \(\mathcal{F}_d(n)\) differ by at most two bit changes.  
When we concatenate the listings of \(\mathcal{F}_d(n)\) for \( d = 1 \) to \( \lfloor \frac{n}{2} \rfloor \), transitioning from the last string in \(\mathcal{F}_d(n)\) to the first string in \(\mathcal{F}_{d-1}(n)^{R}\) requires removing a 1 by Lemma~\ref{lemma-start}, which requires only one bit change. 
Similarly, transitioning from the last string in $  \mathcal{F}_d(n)^{R}  $ to the first string in $  \mathcal{F}_{d-1}(n)  $ requires only one bit change.  \qed
%By Lemma~\ref{}, consecutive strings in $\mathcal{F}_d(n)$ differ by at most two bits. 
%When we concatenate $\mathcal{F}_d(n)$ and $\mathcal{F}_{d-1}(n)$, we simply decrement the last 1 or the first 1, depending on the parity which only requires one bit change.  
%According to Lemma 2, the sequence $\mathcal{F}_d(n)$ ensures that transitions between consecutive strings require at most two bit changes. The sequence $\mathcal{F}_d(n)$ is ordered based on the parity of $\lfloor n/2 \rfloor - d$, reversing when odd and maintaining forward order when even, with parity alternating as $d$ decreases. By the algorithm~\cite{}, the first string of $\mathcal{F}_d(n)$ is derived from the last string of $\mathcal{F}_{d+1}(n)$ by flipping the leftmost 1 (for forward order) or rightmost 1 (for reverse order), preserving the consecutive strings in $F(n)$ differ by at most two bit changes.
\end{proof}

\begin{lemma}
The listing  $\mathcal{F}(n)$ can be generated in constant amortized time per string, using $O(n^2)$ memory. 
\end{lemma}

\begin{proof}
As standard for generation algorithms, the time required to output a string is not included in the analysis.
By Lemma~\ref{lem:CATFd}, each string in the listings $\mathcal{F}_d(n)$ can be generated in constant amortized time per bit, using $O(n^2)$ memory.
The strings in $\mathcal{F}_d(n)^R$ can be produced by simply reversing the print order of each string generated in $\mathcal{F}_d(n)$, and they can clearly be generated in constant amortized time per bit as well, using $O(n^2)$ memory.
\end{proof}

\subsection{Stage 2: Generating binary strings with minimum weight}

In the second stage of our algorithm, we generate a 2-Gray code for the up-steps and horizontal steps, corresponding to the positions in $q$ where $q_i = 0$. 
If a grand Motzkin path has $d$ down-steps (i.e., $d$ 1s in $q$), it must have at least $d$ up-steps to return to the $x$-axis. 
Then when we represent an up-step as 1 and a horizontal step as 0, the set of up-steps and horizontal steps in a grand Motzkin path with air pockets and $d$ down-steps corresponds to the set of binary strings $r$ of length $n-d$ with at least $d$ 1s.

Given $d$ down-steps from the first stage, the second stage generates a 2-Gray code for the set of binary strings $r$ with at least $d$ 1s. 
There are many algorithms that generate 2-Gray codes for binary strings with at least $d$ 1s~\cite{Art4,mutze2023combinatorial,ruskey1996combinatorial,Sav97}. 
For example, {Sawada} et al.~\cite{10.1007/978-3-030-85088-3_15,flipswap} extend the binary reflected Gray code algorithm to generate 2-Gray codes for flip-swap languages. 
A \emph{flip-swap language} with respect to 0 is a set $S$ of binary strings of length $n$ such that $S \cup {1^n}$ is closed under two operations (when applicable): (1) flipping the leftmost 0, and (2) swapping the leftmost 0 with the bit to its right.
\begin{lemma} \cite{10.1007/978-3-030-85088-3_15,flipswap}
The set of binary strings with at least $d$ 1s forms a flip-swap language with respect to 0.
\end{lemma}
Thus, we can use the binary reflected Gray code algorithm from~\cite{10.1007/978-3-030-85088-3_15,flipswap} to generate a 2-Gray code for the set of binary strings $r$ with at least $d$ 1s. 
For example, the algorithm generates the following 2-Gray code for binary strings of length 5 with at least two 1s:
\begin{center}
11111, 01111, 00111, 10111, 10011, 00011, 01011, 
\\11011, 11001, 01001, 10001, 10101, 00101, 01101, \\11101, 11100, 01100, 10100, 11000, 11010, 01010, \\10010, 10110, 00110, 01110, 11110.  \ \ \ \ \ \ \ \ \ \ \ \ \ \ \ \ \ \ \ \ 
\end{center}
Let $\mathcal{B}_d(t)$ denote the listing of strings of $r$ for grand Motzkin path with air pockets of length $t$ with $d$ down-steps generated by the algorithm in~\cite{10.1007/978-3-030-85088-3_15,flipswap}. 

  \begin{lemma}\label{BnFull}~\cite{10.1007/978-3-030-85088-3_15,flipswap}
The listing $\mathcal{B}_d(t)$ always starts with the string $1^t$ and ends with the string $1^{t-1}0$. 
\end{lemma}

\begin{lemma}\label{BnGray}~\cite{10.1007/978-3-030-85088-3_15,flipswap}
Consecutive strings in $\mathcal{B}_d(t)$ differ by at most two bit changes. 
\end{lemma}

\begin{lemma}~\cite{10.1007/978-3-030-85088-3_15,flipswap}
The listing  $\mathcal{B}_d(t)$ can be generated in constant amortized time per string, using $O(n^2)$ memory. 
\end{lemma}

\subsection{Stage 3: Generating integer compositions}

The last stage of our algorithm involves generating all possible combinations of the $d$ down-step positions.
An \emph{integer composition} of a positive integer $s$ is a sequence $a_1, a_2, \ldots, a_k$ of positive integers such that $a_1 + a_2 + \cdots + a_k = s$.
In other words, an integer composition can be considered as an ordered collection of positive integers that sum to $s$.
Each term $a_i$ in the sequence is called a \emph{part}.
For example, the eight integer compositions of $s = 4$ are:
\begin{equation}
\label{list:composition}
4, \ \ 1 +3, \ \ 2 +2, \ \ 3 +1, \ \ 1 +1 +2, \ \ 1 +2 +1, 
\ \ 2 +1+1, \ \ 1 +1 +1 +1.
\end{equation}
Suppose that in stages one and two of our algorithm, we are given $d$ down-steps and $u$ up-steps. 
Then, in order for the grand Motzkin path to land on the $x$-axis, the $d$ down-steps must add up to $u$, which corresponds to all possible integer compositions of $u$ with $d$ parts.

%In~\cite{}, Shao et al. introduced a novel binary representation to represent integer partitions.
%In the proposed \emph{binary representation} for integer partitions, each integer partition of $n$ is represented by a binary string of length $n$.
%Each part $a_k$ of an integer partition is represented by the substring $0^{|a_k|-1}1$, and we concatenate all parts together to represent the integer partition.
%As an example, the seven integer partitions of the integer $n=5$ in~(\ref{list:partition}) can be represented by the following binary strings under our binary representation:

%\begin{center}
%$00001, 00011, 00101, 00111, 01011, 01111, 11111$.
%\end{center}
%Integer partitions can also be represented under the standard integer representation
%(also known as the natural representation), which lists each part as an integer.
%For example, the seven integer partitions of the integer $n=5$ in~(\ref{list:composition}) can be represented by the following lists of integers under the standard integer representation.
%$$\label{list:partition}
%5, \ \ 4 +1, \ \ 3 +2, \ \ 3 +1 +1, \ \ 2 +2 +1, \ \ 2 +1 +1 +1, \ \ 1 +1 +1 +1 +1.
%$$

We adopt a similar idea proposed in~\cite{Shao25} and introduce a
\emph{shorthand binary representation} to represent integer compositions.
In our proposed shorthand binary representation for integer compositions, each integer composition of $s$ is represented by a binary string of length $s-1$.
Each part $a_k$ of an integer composition is represented by the substring $0^{|a_k|-1}1$, except for the last part, which is represented by $0^{|a_k|-1}$, and we concatenate all parts together to represent the integer composition.
As an example, the eight integer compositions of the integer $s=4$ in~(\ref{list:composition}) can be represented by the following binary strings under our shorthand binary representation:
\begin{center}
$000, 100, 010, 001, 110, 101, 011, 111$.
\end{center}
Integer compositions of $s$ with $k$ parts can thus be represented by binary strings of length $s-1$ with weight $k-1$.

\begin{lemma}
The number of integer compositions of $s$ with $k$ parts is $\binom{s - 1}{k - 1}$. 
\end{lemma}

A Gray code is said to be \emph{homogeneous} when the bits between the swapped 0 and 1 are all 0s.
We now prove that a homogeneous 2-Gray code for integer compositions under our shorthand binary representation corresponds to a 2-Gray code for integer compositions under the standard integer representation.

\begin{lemma}\label{composition}
The set of strings that form a homogeneous 2-Gray code for integer compositions in shorthand binary representation, when converted to standard integer representation, forms a 2-Gray code for integer compositions in that representation.
%The set of strings forming a homogeneous 2-Gray code for integer compositions in the shorthand binary representation, when expressed in the standard integer representation, forms a 2-Gray code for integer compositions in the standard integer representation.
\end{lemma}
\begin{proof}
The detailed proof is provided in Appendix C. 
\end{proof}

There are many algorithms that generate homogeneous 2-Gray codes for binary strings of length $n$ with $k$ 1s.
We again adopt the algorithm in~\cite{Hassler24KMS,Hassler24GASCom,Hassler24RAIRO} to generate a homogeneous 2-Gray code for shorthand binary representations of integer compositions.
For example, the algorithm generates the following listing when $n=7$ and $k=2$.
\begin{center}
1100000, 1010000, 0110000, 0101000, 1001000, 0011000, 0010100,
\\1000100, 0100100, 0001100, 0001010, 1000010, 0100010, 0010010,
\\0000110, 0000101, 1000001, 0100001, 0010001, 0001001, 0000011.
\end{center}
The corresponding 2-Gray code for integer compositions are
\begin{center}
1+1+6, \ \  1+\underline{2}+\underline{5}, \ \  \underline{2}+\underline{1}+5, \ \  2+\underline{2}+\underline{4}, \ \  \underline{1}+\underline{3}+4, \ \ \underline{3}+\underline{1}+4, \ \ 3+\underline{2}+\underline{3},
\\\underline{1}+\underline{4}+3, \ \  \underline{2}+\underline{3}+3, \ \  \underline{4}+\underline{1}+3, \ \  4+\underline{2}+\underline{2}, \ \  \underline{1}+\underline{5}+2, \ \ \underline{2}+\underline{4}+2, \ \ \underline{3}+\underline{3}+2,
\\\underline{5}+\underline{1}+2, \ \  5+\underline{2}+\underline{1}, \ \  \underline{1}+\underline{6}+1, \ \  \underline{2}+\underline{5}+1, \ \  \underline{3}+\underline{4}+1, \ \ \underline{4}+\underline{3}+1, \ \ \underline{6}+\underline{1}+1.
\end{center}
The changes on each consecutive integer partitions has been underlined. 
Let $\mathcal{I}_d(s)$ denote the listing of integer compositions of $s$ with $d$ parts under the standard integer representation, generated by the algorithm in~\cite{Hassler24KMS,Hassler24GASCom,Hassler24RAIRO}.
\begin{lemma}\label{IsStart}\cite{Hassler24KMS,Hassler24GASCom,Hassler24RAIRO}
The listing $\mathcal{I}_d(s)$ always starts with the integer composition $\underbrace{1 + \dots + 1}_{d-1} + (s-d+1)$ and ends with the composition $(s-d+1) + \underbrace{1 + \dots + 1}_{d-1}$.
\end{lemma}

\begin{lemma}\label{IsGray}\cite{Hassler24KMS,Hassler24GASCom,Hassler24RAIRO}
Consecutive strings in $\mathcal{I}_d(s)$ differ by at most two symbol changes.
\end{lemma}

\begin{lemma}\cite{Hassler24KMS,Hassler24GASCom,Hassler24RAIRO}
The listing $\mathcal{I}_d(s)$ can be generated in constant amortized time per string using $O(n^2)$ space.
\end{lemma}

\subsection{Assembling the listings}

In previous subsections, we discussed efficient constructions for 2-Gray codes for the sequences $\mathcal{F}(n)$, $\mathcal{B}_d(t)$, and $\mathcal{I}_d(s)$. This subsection describes how to concatenate these listings to produce a 2-Gray code for grand Motzkin paths with air pockets. 
%\hl{Here, we define the \textit{rank} of a string (or listing element) as its position in the generated sequence, counted from 1 (i.e., the first element has rank 1, the second rank 2, etc.)}. 
The \emph{rank} of a string is its position in the generated sequence, starting from 1.
We also use the notation $\mathcal{S}^{-1}$ to refer to the
listing obtained by reversing a listing $\mathcal{S}$. 
Similarly, $\mathcal{S}^{-k}$ represents the listing obtained by reversing the listing $\mathcal{S}$ for $k$ times. 
Thus, when $k$ is even $\mathcal{S}^{-k} = \mathcal{S}$, and when $k$ is odd $\mathcal{S}^{-k}=\mathcal{S}^{-1}$. 
The construction involves the following steps:
\begin{enumerate}
\item Generate the strings in $\mathcal{F}(n)$;
\item For each string $w$ in $\mathcal{F}(n)$, let $d$ be the number of 1s in $w$. Generate the strings in $\mathcal{B}_d(n-d)$. If the rank of $w$ in  $\mathcal{F}(n)$ is odd, for each $v$ in $\mathcal{B}_d(n-d)$, produce a string $\ell$ by replacing the 0s in $w$ with the bits in $v$ and the 1s in $w$ with the symbol D. 
If the rank of $w$ in $\mathcal{F}(n)$ is even, reverse the whole listing $\mathcal{B}_d(n-d)$, that is $\mathcal{B}_d(n-d)^{-1}$; then produce $\ell$ similarly. For the even-rank case, 
%apply a fixed permutation to the bit positions
permute the symbols of each string in the same way in the listing so that 
%rearrange the column of the listing to 
 the first string of $\mathcal{B}_d(n-d)^{-1}$ matches the last string generated in the previous sublisting;
\item For each $\ell$ produced, let $s$ be the number of 1s in $v$. Generate the listing $\mathcal{I}_d(s)$. If the rank of $\ell$ is odd, replace the Ds in $\ell$ with the integer compositions from $\mathcal{I}_d(s)$. If the rank of $\ell$     is even, reverse each integer composition in $\mathcal{I}_d(s)$ individually, that is $\mathcal{I}_d(s)^R$, and replace Ds with the integer compositions from $\mathcal{I}_d(s)^R$.
%If the rank of $\ell$ is even, reverse each string in $\mathcal{I}_d(s)$ and replace the Ds in $\ell$ with the compositions from the reversed $\mathcal{I}_d(s)$.
\end{enumerate}
As an example, suppose we want to generate a Gray code for grand Motzkin paths with air pockets with $n=5$. We first generate the strings
in $\mathcal{F}(5)$:
\begin{align*}
\mathcal{F}(5) &= \mathcal{F}_2(5) \cdot \mathcal{F}_1(5)^{R} \\
&= 10100, 10010, 01010, 01001, 10001, 00101, 00001, 00010, 00100, 01000, 10000.
\end{align*}
Then, in the second step, for each string $w$ in $\mathcal{F}(5)$, we replace the 1s with Ds and the 0s with strings from $\mathcal{B}_d(n-d)$ or its reverse, depending on the parity of its rank.
{\small
\begin{align*}
\mathcal{B}_2(3) &= 111, 011, 101, 110;\\
\mathcal{B}_1(4) &= 1111
,0111
,0011
,1011
,1001
,0001
,0101
,1101
,1100
,0100
,1000, 
1010,\\
& \ \ \ \  0010
,0110
,1110;\\
\mathcal{B}_2(3)^{-1} &= 110, 101, 011, 111;\\
\mathcal{B}_1(4)^{-1} &= 1110
,0110
,0010
,1010
,1000
,0100
,1100
,1101
,0101
,0001
,1001
,1011, \\
& \ \ \ \  0011
,0111
,1111.
\end{align*}}

%permute the symbols of each string in the same way in the listing so that 
%rearrange the column of the listing to 
% the first string of $\mathcal{B}_d(n-d)^{-1}$ matches the last string generated in the previous sublisting;

There are many ways to permute the symbols of each string in the listing to ensure that the first string of $\mathcal{B}_d(n-d)^{-1}$ matches the last string generated in the previous sublisting. 
We maintain a linked list $ m $ that serves as an index mapping for non-down-step positions in the string $ w $ during step two of the Gray code concatenation for grand Motzkin paths. It is initialized as the sorted list of positions where $ w $ has 0s, guiding where bits from $ v $ (in $ \mathcal{B}_d(n-d) $) are inserted to form $ \ell $, with D placed at 1s in $ w $. When transitioning to the next $ w' $, $ m $ is updated by replacing the entry for the position turning from 0 to 1 with the position turning from 1 to 0; additionally, if a down-step is replaced by an up-step or horizontal step (adding a new non-down-step position), the new position is inserted at the beginning of $ m $, and all subsequent indices are incremented by one.
Then each string $w$ in $\mathcal{F}(5)$ becomes the following listings:
%\vspace{-0.95em}
{\small
\begin{align*}
10100, m=\{2, 4, 5\} &= \text{D}1\text{D}11, \text{D}0\text{D}11, \text{D}1\text{D}01, \text{D}1\text{D}10;\\
10010, m=\{2, 3, 5\} &= \text{D}11\text{D}0, \text{D}10\text{D}1, \text{D}01\text{D}1, \text{D}11\text{D}1;\\
01010, m=\{1, 3, 5\}&= 1\text{D}1\text{D}1, 0\text{D}1\text{D}1, 1\text{D}0\text{D}1, 1\text{D}1\text{D}0;\\
01001, m=\{1, 3, 4\} &= 1\text{D}10\text{D}, 1\text{D}01\text{D}, 0\text{D}11\text{D}, 1\text{D}11\text{D};\\
10001, m=\{2, 3, 4\} &= \text{D}111\text{D}, \text{D}011\text{D}, \text{D}101\text{D}, \text{D}110\text{D};\\
00101, m=\{2, 1, 4\} &= 11\text{D}0\text{D}, 01\text{D}1\text{D}, 10\text{D}1\text{D}, 11\text{D}1\text{D};\\
00001, m=\{3, 2, 1, 4\} &= 1111\text{D}, 1101\text{D}, 1001\text{D}, 1011\text{D}, 0011\text{D}, 0001\text{D}, 0101\text{D}, 0111\text{D}, \\& \ \ \ \ 0110\text{D}, 0100\text{D}, 0010\text{D}, 1010\text{D}, 1000\text{D}, 1100\text{D}, 1110\text{D};\\
00010, m=\{3, 2, 1, 5\} &= 111\text{D}0, 110\text{D}0, 100\text{D}0, 101\text{D}0, 001\text{D}0, 010\text{D}0, 011\text{D}0, 011\text{D}1, \\& \ \ \ \  010\text{D}1, 000\text{D}1, 001\text{D}1, 101\text{D}1, 100\text{D}1, 110\text{D}1, 111\text{D}1;\\
00100, m=\{4, 2, 1, 5\} &= 11\text{D}11, 11\text{D}01, 10\text{D}01, 10\text{D}11, 00\text{D}11, 00\text{D}01, 01\text{D}01, 01\text{D}11, \\& \ \ \ \  01\text{D}10, 01\text{D}00, 00\text{D}10, 10\text{D}10, 10\text{D}00, 11\text{D}00, 11\text{D}10;\\
01000, m=\{4, 3, 1, 5\} &= 1\text{D}110, 1\text{D}100, 1\text{D}000, 1\text{D}010, 0\text{D}010, 0\text{D}100, 0\text{D}110, 0\text{D}111, \\& \ \ \ \  0\text{D}101, 0\text{D}001, 0\text{D}011, 1\text{D}011, 1\text{D}001, 1\text{D}101, 1\text{D}111;\\
10000, m=\{4, 3, 2, 5\} &= \text{D}1111, \text{D}1101, \text{D}1001, \text{D}1011, \text{D}0011, \text{D}0001, \text{D}0101, \text{D}0111, \\& \ \ \ \  \text{D}0110, \text{D}0100, \text{D}0010, \text{D}1010, \text{D}1000, \text{D}1100, \text{D}1110.
\end{align*}}
%Note that when $n$ is larger, it might involve cases where we need to rearrange the columns of the listing. 
%For example, when $n=6$, we have 1D110D followed by 1D1D01. In this case, we have to rearrange the fifth and sixth columns of 1D1D01 because $\mathcal{B}_2(4)^{-1}$ starts with 1110.
\begin{algorithm}[t]
\caption{An algorithm to generate grand Motzkin paths with air pockets.}\label{alg:motzkin}
\begin{algorithmic}[1]
\Procedure{GenM}{$n$}
    \State{$(i, j, h, z) \gets (0, 0, 0, 0)$}
%    \State{$i \gets 0$}
%    \State{$j \gets 0$}
%    \State{$h \gets 0$}
    \State{$m \gets [x]_{x\in[1, 2, \dots, n]} \backslash [2y-1]_{y\in [1, 2, \dots, \lfloor\frac{n}{2}\rfloor]}$}
    \For{$d$ from $\lfloor \frac{n}{2} \rfloor$ to $1$} 
        \For{$\alpha \in \mathcal{F}_{d}(n)$}
        \If{$i \neq 0$} {$\alpha \gets \alpha^R$}
        \EndIf
        \State{Switch $m[x]$ and $m[y]$ where $\alpha_x \neq \alpha^\prime_x$ and $\alpha_y  \neq \alpha^\prime_y$}
         \For{$\beta \in \mathcal{B}_{d}(n-d)^{-j}$}
            \If{$|\beta| = 1  \mathrm{~and~} | \beta'| = 1 \mathrm{~and~} z = 0$}
            \State{\textsc{Print} $(0, 0, \ldots, 0)$}
            \State{$z \gets 1$}
            \EndIf
            \For{$\gamma \in \mathcal{I}_{d}(|\beta|)$} 
            \If{$h \neq 0$} {$\gamma \gets \gamma^R$}
            \EndIf
            \State{{\sc Print DecodeMotzkin}($\alpha, \beta, m, \gamma$)}
            \EndFor
            \State{$h \gets \neg h$}
            \State{$\beta' \gets \beta$}
        \EndFor
        \State{$j \gets \neg j$}
        \State{$\alpha^\prime \gets \alpha$}
        \EndFor
        \State{$i \gets \neg i$}
        \If{$i = 0$} {$m \gets [2d-1] \cdot m$}
        \Else { $m \gets [n-2d+2] \cdot m$}
        \EndIf
    \EndFor
\EndProcedure
\end{algorithmic}
\end{algorithm}

In the third step, for each string generated in the second step, we replace the Ds with the parts obtained from the integer compositions in $\mathcal{I}_d(s)$ or its mirror image, depending on the parity of its rank. 
For example, $\mathcal{I}_2(3) = (1+2, 2+1)$ and its mirror image is $(2+1, 1+2)$. 
Note that $\mathcal{I}_1(s)$ and its mirror image are identical (both contain only $s$), and $\mathcal{I}_2(2)$ and its mirror image are both $(1+1)$.
%The resulting listing for grand Motzkin paths with air pockets is thus as follows - note that 
We also add the special string $(0, 0, \ldots, 0)$ between the first pair of consecutive  strings that both have only one single up-step to ensure the listing includes all grand Motzkin paths with air pockets.  
Each grand Motzkin path is represented using a \emph{compact representation}, where a path $(a_1, a_2, \dots, a_n)$ is encoded as a string $b_1 b_2 \cdots b_n$ of length $n$ with $b_i = a_i$ if $a_i \geq 0$ (up-step or horizontal step), and $b_i$ represented as $\overline{|a_i|}$ if $a_i < 0$ (down-step).
\begin{center}
\begin{tabular}{l}
$\overline{1}1\overline{2}11, \overline{2}1\overline{1}11, \overline{1}0\overline{1}11, \overline{1}1\overline{1}01, \overline{1}1\overline{1}10, \overline{1}11\overline{1}0, \overline{1}10\overline{1}1, \overline{1}01\overline{1}1, \overline{2}11\overline{1}1, \overline{1}11\overline{2}1, 1\overline{1}1\overline{2}1, 1\overline{2}1\overline{1}1, $ \\
$0\overline{1}1\overline{1}1, 1\overline{1}0\overline{1}1, 1\overline{1}1\overline{1}0, 1\overline{1}10\overline{1}, 1\overline{1}01\overline{1}, 0\overline{1}11\overline{1}, 1\overline{2}11\overline{1}, 1\overline{1}11\overline{2}, \overline{1}111\overline{2}, \overline{2}111\overline{1}, \overline{1}011\overline{1}, \overline{1}101\overline{1}, $ \\
$\overline{1}110\overline{1}, 11\overline{1}0\overline{1}, 01\overline{1}1\overline{1}, 10\overline{1}1\overline{1}, 11\overline{2}1\overline{1}, 11\overline{1}1\overline{2}, 1111\overline{4}, 1101\overline{3}, 1001\overline{2}, 1011\overline{3}, 0011\overline{2}, 0001\overline{1}, $ \\
$0101\overline{2}, 0111\overline{3}, 0110\overline{2}, 0100\overline{1}, 00000, 0010\overline{1}, 1010\overline{2}, 1000\overline{1}, 1100\overline{2}, 1110\overline{3}, 111\overline{3}0, 110\overline{2}0, $ \\
$100\overline{1}0, 101\overline{2}0, 001\overline{1}0, 010\overline{1}0, 011\overline{2}0, 011\overline{3}1, 010\overline{2}1, 000\overline{1}1, 001\overline{2}1, 101\overline{3}1, 100\overline{2}1, 110\overline{3}1, $ \\
$111\overline{4}1, 11\overline{4}11, 11\overline{3}01, 10\overline{2}01, 10\overline{3}11, 00\overline{2}11, 00\overline{1}01, 01\overline{2}01, 01\overline{3}11, 01\overline{2}10, 01\overline{1}00, 00\overline{1}10, $ \\
$10\overline{2}10, 10\overline{1}00, 11\overline{2}00, 11\overline{3}10, 1\overline{3}110, 1\overline{2}100, 1\overline{1}000, 1\overline{2}010, 0\overline{1}010, 0\overline{1}100, 0\overline{2}110, 0\overline{3}111, $ \\
$0\overline{2}101, 0\overline{1}001, 0\overline{2}011, 1\overline{3}011, 1\overline{2}001, 1\overline{3}101, 1\overline{4}111, \overline{4}1111, \overline{3}1101, \overline{2}1001, \overline{3}1011, \overline{2}0011, $ \\
$\overline{1}0001, \overline{2}0101, \overline{3}0111, \overline{2}0110, \overline{1}0100, \overline{1}0010, \overline{2}1010, \overline{1}1000, \overline{2}1100, \overline{3}1110.$
\end{tabular}
\end{center}

Pseudocode of the algorithm is given by Algorithm~\ref{alg:motzkin}. 
It follows our three-stage algorithm using three nested loops to enumerate strings in \(\mathcal{F}(n)\), \(\mathcal{B}_d(n-d)\), and \(\mathcal{I}_d(s)\), which encode the positions of down-steps, the assignments of up-steps and horizontal steps, and the magnitudes of down-steps, respectively. 
The algorithm maintains a linked list \(m\), which serves as an index mapping for non-down-step positions during the second stage to ensure Gray code properties in the concatenation of \(\mathcal{B}_d(n-d)\) listings. 
It also tracks the last generated string \(\alpha'\) and identifies the positions \(x\) and \(y\) where the current string \(\alpha\) differs from \(\alpha'\) (i.e., \(\alpha_x \neq \alpha'_x\) and \(\alpha_y \neq \alpha'_y\)) to update \(m\) by swapping these indices. 
When reducing the number of down-steps, the algorithm updates \(m\) by prepending the position of the removed down-step position as the first element.
The function \textsc{DecodeMotzkin} is then used to decode the grand Motzkin path with air pockets when given $\alpha \in \mathcal{F}(n)$, $\beta \in \mathcal{B}_d(n-d)$, the linked list $m$, and $\gamma \in \mathcal{I}_d(s)$.
For simplicity, the pseudocode does not include details on handling the special string $(0, 0, \dots, 0)$. 
%The case for including the special string $(0, 0, \cdots , 0)$ is omitted from the pseudocode for simplicity. 
A complete Python implementation of the algorithm is
given in Appendix A.

\begin{theorem}\label{GenMGray}
The algorithm \textsc{GenM} generates a list of all grand Motzkin paths with air pockets in 2-Gray code order.
\end{theorem}

\begin{proof}
The proof of this theorem is provided in Appendix C.
\end{proof}

\begin{theorem}
The algorithm \textsc{GenM} generates a 2-Gray code for grand Motzkin paths with air pockets in constant amortized time per string using $O(n^2)$ space.
\end{theorem}

\begin{proof}
As standard for generation algorithms, the time required to output a string, that is the function \textsc{DecodeMotzkin}, is not part of the analysis.
The index mapping linked list $m$ can be maintained in constant time.
By Lemma~\ref{Fn-Gray}, Lemma~\ref{BnGray}, and Lemma~\ref{IsGray}, the listings $  \mathcal{F}(n)  $, $  \mathcal{B}_d(n-d)  $, and $  \mathcal{I}_d(s)  $ can be generated in constant amortized time per string using $O(n^2)$ space.
The special string $(0, 0, \dots, 0)$ can be included in the listing by tracking the weight of $\beta$, which can be done in constant time.
Thus, \textsc{GenM} generates grand Motzkin paths with air pockets in constant amortized time per string using $O(n^2)$ space.
\end{proof}

\section{Generating grand Dyck paths with air pockets}\label{sec:dyck}

A grand Dyck path with air pockets is a grand Motzkin path with air pockets that uses no horizontal steps. To generate grand Dyck paths with air pockets, we simplify Algorithm~\ref{alg:motzkin} by omitting the second stage, which assigns horizontal steps. 
After selecting \( d \) down-step positions, the remaining \( n-d \) positions are assigned up-steps, and the value \( n-d \) (the number of up-steps) is distributed into the \( d \) down-step magnitudes. 
Pseudocode of the algorithm for grand Dyck paths with air pockets is given by Algorithm~\ref{alg:dyck}. 
It eliminates the nested loop for enumerating strings in \(\mathcal{B}_d(n-d)\) and the linked list \( m \) for index mapping. 
The function \textsc{DecodeDyck} constructs the grand Dyck path with air pockets from inputs \(\alpha \in \mathcal{F}_{d}(n)\), representing down-step positions, and \(\gamma \in \mathcal{I}_{d}(s)\), representing down-step magnitudes. 
For example, Algorithm~\ref{alg:dyck} generates the following grand Dyck paths with air pockets for \( n = 5 \):
\begin{center}
\small
\begin{tabular}{l}
(-4, 1, 1, 1, 1),\ (1, -4, 1, 1, 1),\ (1, 1, -4, 1, 1),\ (1, 1, 1, -4, 1),\ (1, 1, 1, 1, -4),\\
(1, 1, -2, 1, -1),\ (1, 1, -1, 1, -2),\ (1, -1, 1, 1, -2),\ (1, -2, 1, 1, -1),\ (1, -2, 1, -1, 1),\\
(1, -1, 1, -2, 1),\ (-1, 1, 1, -2, 1),\ (-2, 1, 1, -1, 1),\ (-2, 1, 1, 1, -1),\ (-1, 1, 1, 1, -2),\\
(-1, 1, -2, 1, 1),\ (-2, 1, -1, 1, 1).
\end{tabular}
\end{center}
A complete Python implementation is provided in Appendix A. 

\begin{theorem}\label{GenDGray}
The algorithm \textsc{GenD} generates a list of all grand Dyck paths with air pockets in 2-Gray code order.
\end{theorem}

\begin{proof}
The proof of this theorem is provided in Appendix C.
\end{proof}
\begin{algorithm}[t]
\caption{An algorithm to generate grand Dyck paths with air pockets.}\label{alg:dyck}
\begin{algorithmic}[1]
\Procedure{GenD}{$n$}
%    \State{$i \gets 0$}
%    \State{$j \gets 0$}
    \State{$(i, j) \gets (0, 0)$}
    \For{$d$ from $\lfloor \frac{n}{2} \rfloor$ to $1$} 
        \For{$\alpha \in \mathcal{F}_{d}(n)$}
            \If{$i \neq 0$} {$\alpha \gets \alpha^R$}
            \EndIf
            \For{$\gamma \in \mathcal{I}_{d}(n-d)$}
            \If{$j \neq 0$} {$\gamma \gets \gamma^R$}
            \EndIf
            \State{\textsc{Print} {\sc DecodeDyck}($\alpha, \gamma$)}
            \EndFor
            \State{$j \gets \neg j$}
        \EndFor
        \State{$i \gets \neg i$}
    \EndFor
\EndProcedure
\end{algorithmic}
\end{algorithm}

\begin{theorem}\label{GenDCAT}
The algorithm \textsc{GenD} generates a 2-Gray code for grand Dyck paths with air pockets in constant amortized time per string using $O(n^2)$ space.
\end{theorem}

\begin{proof}
As standard for generation algorithms, the time required to output a string, that is the function \textsc{DecodeDyck}, is not part of the analysis.
By Lemma~\ref{Fn-Gray} and Lemma~\ref{IsGray}, the listings
$  \mathcal{F}(n)  $ and $  \mathcal{I}_d(s)  $ can be generated in constant amortized time per string using $O(n^2)$ space.
Thus, \textsc{GenD} generates grand Dyck paths with air pockets in constant amortized time per string using $O(n^2)$ space.\qed\end{proof}

%For example, consider the grand Motzkin path with air pockets represented by the tuple $m = (1, 1, -1, 1, -4, 1, -1, 1, 1, 1, 1, 1, -4, 1)$.

\begin{comment}
\begin{figure}[h]
\centering
\begin{tikzpicture}[scale=0.6]
    % Grand Motzkin Path: [1, 1, -1, 1, -4, 1, -1, 1, 1, 1, 1, 1, -4, 1]
    % Path: (0,0) -> (1,1) -> (2,2) -> (3,1) -> (4,2) -> (5,-2) -> (6,-1) -> (7,-2) -> (8,-1) -> (9,0) -> (10,1) -> (11,2) -> (12,3) -> (13,-1) -> (14,0)
    % y-values: 0, 1, 2, 1, 2, -2, -1, -2, -1, 0, 1, 2, 3, -1, 0
    % Min y = -2, Max y = 3
    \draw[->] (0,0) -- (15,0) node[right] {$x$}; % x-axis, length 15 for 14 steps
    \draw[->] (0,-2.5) -- (0,4.5) node[above] {$y$}; % y-axis, from -2.5 to 4.5
    % x-axis ticks and labels
    \foreach \x in {1,2,...,14} {
        \draw (\x,0.1) -- (\x,-0.1) node[below] {\x};
    }
    \draw (0,0.1) -- (0,-0.1) node[below left] {$O$}; % Origin labeled as O
    % y-axis ticks and labels
    \foreach \y in {-2,-1,0,1,2,3,4} {
        \draw (0.1,\y) -- (-0.1,\y) node[left] {\y};
    }
    % Path
    \draw[black, line width=0.5mm] (0,0) -- (1,1) -- (2,2) -- (3,1) -- (4,2) -- (5,-2) -- (6,-1) -- (7,-2) -- (8,-1) -- (9,0) -- (10,1) -- (11,2) -- (12,3) -- (13,-1) -- (14,0);
    \foreach \x/\y in {0/0, 1/1, 2/2, 3/1, 4/2, 5/-2, 6/-1, 7/-2, 8/-1, 9/0, 10/1, 11/2, 12/3, 13/-1, 14/0} \draw[draw=black, fill=white] (\x,\y) circle (3pt); % Hollow circles
\end{tikzpicture}
\caption{An example of grand Dyck path with air pockets for $m = (1, 1, -1, 1, -4, 1, -1, 1, 1, 1, 1, 1, -4, 1)$.}
\label{fig:Dyck}
\end{figure}

\end{comment}

\section{Enumeration} \label{sec:enumeration}

This section provides the enumeration formulae for grand Motzkin paths with air pockets and grand Dyck paths with air pockets. 

\begin{theorem}\label{enumMot}
The number of grand Motzkin paths with air pockets of length \( n \) is given by:
\begin{equation}\nonumber
1 + \sum_{d=1}^{\lfloor \frac{n}{2} \rfloor} \binom{n-d+1}{d} \sum_{i=0}^{n-2d} \binom{n-d}{i} \binom{n-d-i-1}{d-1}.
\end{equation}
\end{theorem}
 
\begin{proof}
For a path with \( d \) down-steps (\( 1 \leq d \leq \lfloor \frac{n}{2} \rfloor \)):
\begin{enumerate}
    \item Choose \( d \) positions for the down-steps among the \( n \) total steps, ensuring no two down-steps are consecutive. The number of ways to select \( d \) non-consecutive positions is the number of binary strings of length \( n \) with \( d \) 1s and no consecutive 1s, given by \( \binom{n-d+1}{d} \);
    \item For the remaining \( n-d \) non-down-step positions, assign \( i \) of them as horizontal steps (and thus \( n-d-i \) as up-steps). The number of ways to choose \( i \) positions is \( \binom{n-d}{i} \), where \( i \) ranges from 0 to \( n-2 \).  Note that there must be at least $  d  $ up-steps since the path must return to the $x$-axis, which implies $  n - d - i \geq d  $ and thus $  0 \leq i \leq n - 2d  $;
    \item Assign magnitudes to the \( d \) down-steps such that their sum is \( n-d-i \) (the number of up-steps). The number of ways to distribute \( n-d-i \) into \( d \) positive integers is \( \binom{n-d-i-1}{d-1} \) by the shorthand binary representation of integer compositions.
\end{enumerate}
For a fixed \( d \), summing over all possible numbers of horizontal steps \( i \) from 0 to \( n-2d \) gives:
\[
\sum_{i=0}^{n-2d} \binom{n-d}{i} \binom{n-d-i-1}{d-1}.
\]
Multiplying by the number of ways to choose down-step positions and summing over \( d \) from 1 to \( \lfloor \frac{n}{2} \rfloor \), then adding one for the path with only horizontal steps:
\[
1 + \sum_{d=1}^{\lfloor n/2 \rfloor} \binom{n-d+1}{d} \sum_{i=0}^{n-2d} \binom{n-d}{i} \binom{n-d-i-1}{d-1}.
\]
This matches the formula in the theorem. \qed
\end{proof}

\begin{theorem}\label{EnumDyck}
For \(n \ge 2d\), the number of grand Dyck paths with air pockets of length \(n\) is given by:
\[
\sum_{d=1}^{\lfloor \frac{n}{2} \rfloor} \binom{n-d+1}{d}  \binom{n-d-1}{d-1}.
\]
\end{theorem}

\begin{proof}
The proof of this theorem is provided in Appendix C.
\end{proof}

A complete Python implementation to compute the closed-form formula for the number of grand Motzkin paths and grand Dyck paths with air pockets is given in Appendix B.

\section*{Acknowledgments.}
This research is supported by the Macao Polytechnic University research grant (RP/FCA-02/2022 and fca.6a16.fe65.7). % and the Macao Science and Technology Development Fund (Project code: FDCT/0033/2023/RIA1).
The research of Lin Chen is supported by the National Natural Science Foundation of China under Grant No. 62172455.

\newpage

\bibliographystyle{abbrv}
\bibliography{myrefs}     

\newpage
\appendix
\input{appendix/ap_code}
\clearpage
\input{appendix/enmu}

\clearpage

\section{Proofs of Lemma 11 and Theorems 1, 3 and 6}\label{sec:gens_opt}
\input{appendix/appendix_fib}

\clearpage

\end{document}

%% file: appendix/ap_code.tex
\section{Python code to generate grand Motzkin
paths and grand Dyck paths with air pockets}\label{sec:pyccode}
{
\tiny
\lstset{style=mystyle}
\begin{lstlisting}[language=python]
from math import ceil

class IntComp:
    def __init__(self, n, k, no_consecutive=False):
        self.n = n
        self.k = k
        self.nc = no_consecutive

        if not no_consecutive:
            self.s = [i for i in range(k+1)]
            self.c = [1 for _ in range(k+1)]
            self.d = [0 for _ in range(k+1)]
            self.p = [[0 for _ in range(n+1)] for _ in range(k+1)]
            for i in range(1, k+1):
                self.p[i][i] = 1
            self.s += [n+1]

        else:
            self.s = [2*i-1 for i in range(k+1)]
            self.c = [1 for _ in range(k+1)]
            self.d = [0 for _ in range(k+1)]
            self.p = [[0 for _ in range(n+1)] for _ in range(2*k+1)]
            for i in range(1, k+1):
                self.p[2*i-1][2*i-1] = 1
            self.s[0] = -1
            self.s += [n+2]

        self.idx = []

    def next(self):
        for i in range(1, self.k+1):
            if self.d[i]:
                if self.s[i] == i and not self.nc or self.s[i] == 2*i-1 and self.nc:
                    self.d[i] = 0
                else:
                    j = max(self.s[i-1]+1, i) if not self.nc else max(self.s[i-1]+2, 2*i-1)

                    self.idx = [self.s[i]-1, j-1]
                    self.s[i] = j
                    self.p[i][j] = self.c[i]
                    for t in range(1, i):
                        self.d[t] = 1
                        self.c[t] += 1
                        self.p[t][self.s[t]] = self.c[t]
                    return True

            for j in range(self.s[i]+1, self.s[i+1] - (0 if not self.nc else 1)):
                if self.p[i][j] != self.c[i]:

                    self.idx = [self.s[i]-1, j-1]
                    self.s[i] = j
                    self.p[i][j] = self.c[i]
                    for t in range(1, i):
                        self.d[t] = 1
                        self.c[t] += 1
                        self.p[t][self.s[t]] = self.c[t]
                    return True
        return False

    def output(self):
        res = []
        j = 1
        for i in range(1, self.n+1):
            if self.s[j] != i:
                res += [0]
            else:
                res += [1]
                j += 1
        return res

def brgc_k(n, k, d):
    a = [0]*(n+1)
    a[n] = 1
    # m = 1

    def Gen(t, z, w):
        if t < 1 or t + w <= k:
            if t < 1:
                yield a[1:n+1], w
            else:
                yield [1]*t+a[t+1:n+1], w+t

        else:
            if not (z % 2):
                a[t] = 1
                yield from Gen(t - 1, 0, w + 1)
                a[t] = 0
                yield from Gen(t - 1, 1, w)
            else:
                a[t] = 0
                yield from Gen(t - 1, 0, w)
                a[t] = 1
                yield from Gen(t - 1, 1, w + 1)
                a[t] = 0

    yield from Gen(n, 1 - d, 0)

def grand_dyck(n):
    d1, d2 = 1, 1
    for u in range(n-1, ceil(n/2)-1, -1):
        ic1 = IntComp(n, n-u, True)
        while True:
            a = ic1.output()
            ic2 = IntComp(u-1, n-u-1)
            while True:
                b = ic2.output()
                yield a, d1, b, d2
                if not ic2.next():
                    break
            d2 = 1 - d2
            if not ic1.next():
                break
        d1 = 1 - d1

def decode_gd(a, d1, b, d2):
    res = []
    down_list = []
    if d2:
        down = 0
        for i in b:
            down -= 1
            if i == 1:
                down_list.append(down)
                down = 0
        down -= 1
        down_list.append(down)
    else:
        down = -1
        for i in reversed(b):
            if i == 1:
                down_list.append(down)
                down = 0
            down -= 1
        down_list.append(down)

    idx = 0
    for i in a if d1 else reversed(a):
        if i == 0:
            res += [1]
        else:
            res += [down_list[idx]]
            idx += 1
    return res

def grand_motzkin(n):
    du = 1
    dd = 1
    di = 1
    u_ = 0
    s_uh = None

    for uh in range(ceil(n/2), n):
        d = n - uh
        d_dist_ic = IntComp(n, d, True)
        while True:
            d_dist = d_dist_ic.output()

            if s_uh is None:
                s_uh = []
                for i in range(len(d_dist)):
                    if d_dist[i] == 0:
                        s_uh.append(i)
            elif len(d_dist_ic.idx) > 0:
                for i in range(len(s_uh)):
                    if s_uh[i] == d_dist_ic.idx[1] and dd == 1 or s_uh[i] == n-d_dist_ic.idx[1]-1 and dd == 0:
                        s_uh[i] = d_dist_ic.idx[0] if dd else n-d_dist_ic.idx[0]-1

            uh_dist_brgc = brgc_k(uh, d, du)

            for uh_dist, u in uh_dist_brgc:

                # insert 0*n
                if u_ == 1 and u == 1:
                    yield -1, n, -1, -1, -1, -1
                    u_ = -1
                u_ = u if u_ != -1 else u_

                d_comp_ic = IntComp(u-1, d-1)
                while True:
                    d_comp = d_comp_ic.output()
                    yield d_dist, dd, uh_dist, s_uh, d_comp, di

                    if not d_comp_ic.next():
                        break
                di = 1 - di

            du = uh_dist[-1]

            if not d_dist_ic.next():
                break
        dd = 1 - dd
        s_uh = [2*d - 2 if dd == 1 else n - 2*d + 1] + s_uh

def decode_gm(d_dist, dd, uh_dist, s_uh, d_comp, di):
    if d_dist == -1:
        return [0]*dd
    res = []
    down_list = []
    if di:
        down = 0
        for i in d_comp:
            down -= 1
            if i == 1:
                down_list.append(down)
                down = 0
        down -= 1
        down_list.append(down)
    else:
        down = -1
        for i in reversed(d_comp):
            if i == 1:
                down_list.append(down)
                down = 0
            down -= 1
        down_list.append(down)

    idx_i = 0
    d_dist = d_dist if dd else list(reversed(d_dist))

    for i in range(len(d_dist)):
        if d_dist[i] == 0:
            res += [uh_dist[s_uh.index(i)]]
        else:
            res += [down_list[idx_i]]
            idx_i += 1
    return res

if __name__ == '__main__':
    print('Input n: ')
    n = int(input())
    print('allow horizontal-steps? 1 yes, 0 no')
    is_gmp = int(input())

    is_gmp = True if is_gmp == 1 else False
    gen_func = grand_motzkin if is_gmp else grand_dyck
    decode_func = decode_gm if is_gmp else decode_gd

    cnt = 0
    for ap in gen_func(n):
        cnt += 1
        curr = decode_func(*ap)
        print('%d: %s' % (cnt, str(curr)))

    print('Total: %d' % cnt)

\end{lstlisting}
}

%% file: appendix/enmu.tex
\section{Python code to enumerate the numbers of grand Motzkin paths and grand Dyck paths with air pockets}

{
\tiny
\lstset{style=mystyle}
\begin{lstlisting}[language=python]
import math

def comb(n, k):
    try:
        return math.comb(n, k)
    except ValueError:
        return 0

gdp = lambda n: sum([comb(n - k + 1, k) * comb(n - k - 1, k - 1) for k in range(1, n // 2 + 1)])
gmp = lambda n: sum(
    [comb(n - k + 1, k) * sum([comb(n - k, n - k - i) * comb(n - k - i - 1, k - 1) for i in range(n - 2 + 1)]) for k in range(1, n // 2 + 1)]) + 1

if __name__ == "__main__":
    print("Grand Dyck paths or Grand Motzkin paths with air pockets? (1/2)")
    this_type = input()
    assert this_type in ["1", "2"]
    if this_type == "1":
        func = gdp
        path_type = "grand Dyck paths"
    else:
        func = gmp
        path_type = "grand Motzkin paths"

    print("Enter n:")
    n = int(input())
    result = func(n)
    print(f"the number of {path_type} with air pockets is: {result}")

\end{lstlisting}
}

%% file: appendix/appendix_fib.tex
\begin{comment}
\begin{lemma}
\label{consecutive_upstep_paths}
In the 2-Gray code constructed by the three-stage algorithm, there exist two consecutive grand Motzkin paths with air pockets, each containing exactly one up-step, and the all-zero string \( (0, 0, \dots, 0) \) is positioned between them.
\end{lemma}
\begin{proof}
Stage 1: By Lemma~\ref{lemma-start}, \( F_1(n) \) includes \( q = 0^{n-1}1 \) (with \( d = 1 \)).

Stage 2: By Lemma~\ref{BnFull}, \( B_1(n-1) \) starts with \( 1^{n-1} \) and ends with \( 1^{n-2}0 \), listing all binary strings of length \( n-1 \) with weight at least 1.
By Lemma~\ref{BnGray} and \cite{}, consecutive strings in \( B_1(n-1) \) differ by at most two bit changes. Select \( r = 010^{n-3} \) (weight 1, up-step at position 2) and its consecutive \( r' = 0010^{n-4} \) (weight 1, up-step at position 3), differing by two bits (position 2 changes from 1 to 0, position 3 from 0 to 1).
    
Stage 3: By Lemma~\ref{IsGray}, \( \mathcal{I}_1(1) = 1 \) assigns the down-step value as $-1$.

Construction: With \( q = 0^{n-1}1 \) and \( r = 010^{n-3} \), the path is \( (0, 1, 0, \dots, 0, -1) \). With \( q = 0^{n-1}1 \) and \( r' = 0010^{n-4} \), the path is \( (0, 0, 1, 0, \dots, 0, -1) \).

Inserting the all-zero string $(0,0,\dots,0)$ between
$(0,1,0,\dots,0,-1)$ and $(0,0,1,0,\dots,0,-1)$ maimtains the 2-Gray code. \qed
\end{proof}
\end{comment}

\begin{proof}[Lemma \ref{composition}]
It is clear that integer compositions in the shorthand binary representation are in bijection with integer compositions in the standard integer representation.
In the remainder of the proof, we demonstrate that a homogeneous 2-bit change in the shorthand binary representation corresponds to a two-symbol change in the standard integer representation.

In a homogeneous 2-bit change, a 1 at position \( j \) moves to position \( i \), with all intermediate bits being 0, affecting only the two bits at positions \( i \) and \( j \). 
For a composition with \( d \) parts, this transition corresponds to adjusting two adjacent parts in the standard representation. 
Thus, a homogeneous 2-Gray code in the shorthand binary representation induces a 2-Gray code in the standard integer representation, as only two adjacent parts change. \qed
\end{proof}

\hfill 

\begin{proof}[Theorem \ref{GenMGray}]
By Lemma~\ref{Fn-Gray}, Lemma~\ref{BnGray}, and Lemma~\ref{IsGray}, consecutive strings in $\mathcal{F}(n)$, $\mathcal{B}_d(n-d)$, and $\mathcal{I}_d(s)$ differ by at most two symbol changes and include all possible arrangements of down-steps (via $\mathcal{F}(n)$ and $\mathcal{I}_d(s)$), up-steps, and horizontal steps (via $\mathcal{B}_d(n-d)$) for grand Motzkin paths with air pockets.
We now examine the concatenation of these listings, breaking them down into the following cases:
\begin{itemize}
\item Concatenation of listings of $\mathcal{I}_d(|\beta|)$ where $\beta \in \mathcal{B}_d(n-d)$ is updated: W.L.O.G., consider the concatenation of the listings $\mathcal{I}_d(|\beta'|)$ and $\mathcal{I}_d(|\beta|)^R$, where $|\beta'|$ may differ from $|\beta|$ and $\beta'$ immediately precedes $\beta$ in $\mathcal{B}_d(n-d)$.
Note that consecutive strings in $\mathcal{B}_d(n-d)$ differ by either a swap between a position with 1 and a position with 0 or a single bit change~\cite{10.1007/978-3-030-85088-3_15,flipswap}, so $|\beta'|$ and $|\beta|$ differ by at most one.
In the case where $\beta'$ and $\beta$ differ by a swap, $|\beta'| = |\beta|$, and thus alternating the print order of each string in $\mathcal{I}_d(|\beta|)$ (by Lemma~\ref{IsStart}) ensures that the last string of $\mathcal{I}_d(|\beta'|)$ matches the first string of $\mathcal{I}_d(|\beta|)^R$.
Thus, in this case, there are two bit changes resulting from the swap from $\beta'$ to $\beta$.
In the case of a single bit change between $\beta'$ and $\beta$, we update the first string of $\mathcal{I}_d(|\beta|)^R$ by incrementing or decrementing the largest part by one, which results in one symbol change between the last string of $\mathcal{I}_d(|\beta'|)$ and the first string of $\mathcal{I}_d(|\beta|)^R$.
Thus, in this case, there are a total of two bit changes: one from the change between $\beta'$ and $\beta$, and one between the last string of $\mathcal{I}_d(|\beta'|)$ and the first string of $\mathcal{I}_d(|\beta|)^R$.

\item Concatenation of listings of $\mathcal{B}_d(n-d)$ where $\alpha \in \mathcal{F}_d(n)$ is updated: W.L.O.G., consider the concatenation of the listings $\mathcal{B}_d(n-d)$ and $\mathcal{B}_d(n-d)^R$. By Lemma~\ref{BnFull}, alternating the reversal of $\mathcal{B}_d(n-d)$ and using the index mapping linked list $m$ ensures that the last string of $\mathcal{B}_d(n-d)$ matches the first string of $\mathcal{B}_d(n-d)^R$, while $\mathcal{I}_d(|\beta|)$ remains unchanged. Thus, in this case, there are two bit changes resulting from the update of $\alpha \in \mathcal{F}_d(n)$.

\item Concatenation of listings of $\mathcal{F}_d(n)$ where $d$ is updated: In this case, $d$ is decremented. W.L.O.G., consider the concatenation of the listings $\mathcal{F}_d(n)$ and $\mathcal{F}_{d-1}(n)^R$.
In this case, we replace a down-step with an up-step and increment the largest part of $\mathcal{I}_d(|\beta|)$ by one, resulting in a total of two symbol changes.
\end{itemize}
Finally, we insert the special string $(0, 0, \dots, 0)$ between the first pair of consecutive strings in $\mathcal{B}_1(n)$ that have only one up-step.
Such a pair of consecutive strings exists in $\mathcal{B}_1(n)$~\cite{10.1007/978-3-030-85088-3_15,flipswap}, and this insertion clearly maintains the 2-Gray code property, as the special string differs from both the previous and next strings by two symbol changes.
Thus, the resulting listing forms a 2-Gray code.
\qed
\end{proof}

\hfill

\begin{proof}[Theorem \ref{GenDGray}]
By Lemma~\ref{Fn-Gray} and Lemma~\ref{IsGray}, consecutive strings in $\mathcal{F}(n)$ and $\mathcal{I}_d(s)$ differ by at most two symbol changes and include all possible arrangements of down-steps (via $\mathcal{F}(n)$ and $\mathcal{I}_d(s)$) for grand Dyck paths with air pockets.
We now examine the concatenation of these listings, breaking them down into
the following two cases:
\begin{itemize}
\item Concatenation of listings of $\mathcal{I}_d(n-d)$ where $\alpha \in \mathcal{F}_d(n)$ is updated: W.L.O.G., consider the concatenation of the listings $\mathcal{I}_d(n-d)$ and $\mathcal{I}_d(n-d)^R$, alternating the print order of each string in $\mathcal{I}_d(n-d)$ (by Lemma~\ref{IsStart}) ensures that the last string of $\mathcal{I}_d(n-d)$ matches the first string of $\mathcal{I}_d(n-d)^R$.
Thus, in this case, there are two bit changes resulting from from the update of $\alpha \in \mathcal{F}_d(n)$.
\item Concatenation of listings of $\mathcal{F}_d(n)$ where $d$ is updated: In this case, $d$ is decremented. W.L.O.G., consider the concatenation of the listings $\mathcal{F}_d(n)$ and $\mathcal{F}_{d-1}(n)^R$.
In this case, we replace a down-step with an up-step and increment the largest part of $\mathcal{I}_d(n-d)$ by one, resulting in a total of two symbol changes.
\end{itemize}
Thus, the resulting listing forms a 2-Gray code.\qed
\end{proof}

\hfill 

%\begin{proof}[Theorem \ref{GenDCAT}]
%As standard for generation algorithms, the time required to output a string, that is the function \textsc{DecodeDyck}, is not part of the analysis.
%By Lemma~\ref{Fn-Gray} and Lemma~\ref{IsGray}, the listings
%$  \mathcal{F}(n)  $ and $  \mathcal{I}_d(s)  $ can be generated in %constant amortized time per string using $O(n^2)$ space.
%Thus, \textsc{GenD} generates grand Dyck paths with air pockets in constant amortized time per string using $O(n^2)$ space.\qed
%\end{proof}

\begin{proof}[Theorem \ref{EnumDyck}]
For a path with \( d \) down-steps (\( 1 \leq d \leq \lfloor \frac{n}{2} \rfloor \)):
\begin{enumerate}
    \item Choose \( d \) positions for the down-steps among the \( n \) total steps, ensuring no two down-steps are consecutive. The number of ways to select \( d \) non-consecutive positions is the number of binary strings of length \( n \) with \( d \) 1s and no consecutive 1s, given by \( \binom{n-d+1}{d} \);
    \item The remaining \( n-d \) non-down-step positions are up-steps. Assign magnitudes to the \( d \) down-steps such that their sum is \( n-d \) (the number of up-steps). The number of ways to distribute \( n-d \) into \( d \) positive integers is \( \binom{n-d-1}{d-1} \) by the shorthand binary representation of integer compositions.
    Note that there must be at least $  d  $ up-steps, since the path must return to the x-axis, which implies $  n - d \geq d  $ and $n \geq 2d$. 
\end{enumerate}
For a fixed \( d \), the number of grand Dyck paths with \( d \) down-steps is the product:
\[
\binom{n-d+1}{d} \binom{n-d-1}{d-1}.
\]
Summing over \( d \) from 1 to \( \lfloor \frac{n}{2} \rfloor \) gives:
\[
\sum_{d=1}^{\lfloor \frac{n}{2} \rfloor} \binom{n-d+1}{d} \cdot \binom{n-d-1}{d-1}.
\]
This matches the formula in the theorem. \qed
\end{proof}